\documentclass[12pt]{article} 

\usepackage{amsmath}
\usepackage{amsthm}
\usepackage{amsfonts}
\usepackage{mathrsfs}
\usepackage{stmaryrd}
\usepackage{setspace}
\usepackage{fullpage}
\usepackage{amssymb}
\usepackage{breqn}
\usepackage{enumitem}
\usepackage{bbold} 
\usepackage{authblk}
\usepackage{comment}
\usepackage{hyperref}
\usepackage{pgf,tikz}
\usepackage{graphicx}
\usetikzlibrary{decorations.pathreplacing,arrows}
\tikzstyle{vertex}=[circle,draw=black,fill=black,inner sep=0,minimum size=3pt,text=white,font=\footnotesize]

\newtheorem{thm}{Theorem}[section]

\newtheorem{lemma}[thm]{Lemma}

\newtheorem{proposition}[thm]{Proposition}

\newtheorem{corollary}[thm]{Corollary}

\newtheorem{clm}[thm]{Claim}

\newtheorem*{lemma*}{Lemma}
\newtheorem*{proposition*}{Proposition}
\newtheorem*{theorem*}{Theorem}

\newcommand\ex{\ensuremath{\mathrm{ex}}}

\newcommand\cD{{\mathcal D}}

\newcommand\cH{{\mathcal H}}

\newcommand\cN{{\mathcal N}}

\newcommand\cT{{\mathcal T}}

\newcommand{\ignore}[1]{}

\title{On degree powers and counting stars in $F$-free graphs}
\author{Dániel Gerbner\\ \small Alfr\'ed R\'enyi Institute of Mathematics, HUN-REN\\
\small \texttt{gerbner.daniel@renyi.hu}}
\date{}

\begin{document}

\maketitle

\begin{abstract} Given a positive integer $r$ and  a graph $G$ with degree sequence $d_1,\dots,d_n$, we define $e_r(G)=\sum_{i=1}^n d_i^r$. We let $\ex_r(n,F)$ be the largest value of $e_r(G)$ if $G$ is an $n$-vertex $F$-free graph. We show that if $F$ has a color-critical edge, then $\ex_r(n,F)=e_r(G)$ for a complete $(\chi(F)-1)$-partite graph $G$ (this was known for cliques and $C_5$). We obtain exact results for several other non-bipartite graphs and also determine $\ex_r(n,C_4)$ for $r\ge 3$. We also give simple proofs of multiple known results.

Our key observation is the connection to $\ex(n,S_r,F)$, which is the largest number of copies of $S_r$ in $n$-vertex $F$-free graphs, where $S_r$ is the star with $r$ leaves. We explore this connection and apply methods from the study of $\ex(n,S_r,F)$ to prove our results. We also obtain several new results on $\ex(n,S_r,F)$.



\end{abstract}

\section{Introduction}

Given a positive integer $r$ and  a graph $G$ with degree sequence $d_1,\dots,d_n$, we define $e_r(G)=\sum_{i=1}^n d_i^r$.
This is a well-studied graph parameter, especially in chemical graph theory, see the survey \cite{agmm}. There the quantity $\sum_{i=1}^n d_i^{\alpha}$ is called the \textit{general zeroth order Randi\'c index}, but the names \textit{first general Zagreb
index} and \textit{variable first Zagreb index} are also used in the literature. Section 2.4 of the survey \cite{agmm} deals with the maximum value of this index among $n$-vertex $F$-free graphs. This line of research was initiated by Caro and Yuster \cite{cy}. 

Let us denote by $\ex_r(n,F)$ the largest value of $e_r(G)$ among $n$-vertex $F$-free graphs $G$. Then $\ex_1(n,F)=2\ex(n,F)$, where $\ex(n,F)$ is the \textit{Tur\'an number}, the largest number of edges in $n$-vertex $F$-free graphs. The research on $\ex(n,F)$ was initiated by Turán \cite{T}, who showed that $\ex(n,K_{k+1})=|E(T(n,k))|$, where $T(n,k)$ is the complete $k$-partite graph with parts of order $\lfloor n/k\rfloor$ or $\lceil n/k\rceil$.

Another well-studied topic is called \textit{generalized Turán problem} and deals with counting copies of another subgraph $H$. Let $\cN(H,G)$ denote the number of copies of $H$ in $G$, and $\ex(n,H,F)$ denote the largest value of $\cN(H,G)$ among $n$-vertex $F$-free graphs $G$. The first generalized Turán result is due to Zykov \cite{zykov}, while the systematic study of these problems has only recently been initiated by Alon and Shikhelman \cite{AS}, and has attracted several researchers since. 

Regarding all the parameters mentioned above, we say that an $n$-vertex $F$-free graph is \textit{extremal} if it has the largest possible value of that parameter. Caro and Yuster \cite{cy} observed that for $\ex_r(n,K_{k+1})$ each extremal graph is a complete $k$-partite graph $T$ but not necessarily the Tur\'an graph. Finding the optimal orders of the classes is a problem we are unable to deal with in this generality; such optimization was carried out in \cite{bolnik,lish,cnr,brosid} in different settings. In this paper we avoid this and consider a problem essentially solved if we show that a complete $k$-partite graph (or a complete $k$-partite graph with some additional edges) is extremal.

Observe that for the star $S_r$ with $r$ leaves, we have $\cN(S_r,G)=\sum_{i=1}^n \binom{d_i}{r}$. Therefore, the terms containing $d_i$ in $\cN(S_r,G)$ and $e_r(G)$ are polynomials of the same degree. They differ in the multiplicative constant $1/r!$, and a lower order term, if $d_i$ is sufficiently large. Note that $d_i$ can be smaller than $r$, in which case it gives a term 0 to $\cN(S_r,G)$ and a constant term to $e_r(G)$. However, if $F$ is not a star, then both parameters are $\Omega(n^r)$, as shown by the $F$-free graph $S_{n-1}$. Therefore, the vertices with $d_i=O(1)$ contribute a lower order term anyway. 

If $F$ is a star $S_p$, then it is easy to see that for the two parameters the extremal graph is the same (any $(p-1)$-regular $n$-vertex graph if $n(p-1)$ is even, and any $n$-vertex graph with one vertex of degree $p-2$ and $n-1$ vertices of degree $p-1$ if $n(p-1)$ is odd).
This implies the following.

\begin{proposition}\label{csill} 
We have $\ex(n,S_r,F)=\left(\frac{1}{r!}+o(1)\right)\ex_r(n,F)$. Moreover, the same graphs are asymptotically extremal for both functions.
\end{proposition}

The equivalence of multiple old results are implied by this connection, and also some new results on $\ex_r(n,F)$. We will discuss them in Section \ref{two}.

There are several papers in both settings that focus on obtaining exact results, which Proposition \ref{csill} cannot give. However, we can show a stronger connection between degree powers and counting stars.

\begin{proposition}\label{exac}
    There are positive numbers $w_p=w_p(r)$ such that $e_r(G)=\sum_{p=1}^r w_p\cN(S_p,G)$.
\end{proposition}

Let us state new results that are implied by the above two propositions and earlier results. Let $H(s-1,n)$ denote the graph which consists of $s-1$ vertices of degree $n-1$ and $n-s+1$ vertices of degree $s-1$. In other words, $H(s-1,n)$ is the join of $K_{s-1}$ and the empty graph on $n-s+1$ vertices. Let $G_0$ be a graph on $n-s+1$ vertices with girth at least 5 such that each vertex of $G_0$ has degree $t-1$ except for at most one vertex, which has degree $t-2$. It is a simple corollary of the Erd\H os-Sachs theorem \cite{ersa} that such a graph exists if $n$ is sufficiently large, see e.g. Lemma 2.2 in \cite{dahi2}. Let $H'(s-1,t-1,n)$ denote the graph we obtain from $H(s-1,n)$ by adding $G_0$ to the independent set of order $n-s+1$. 

\begin{proposition}\label{kovik}
\textbf{(i)} If $r< k$, then $\ex(n,S_r,K_{k+1})=(1+o(1))\cN(S_r,T(n,k))$. 


   \textbf{(ii)} If $1<s\le t<r$ and $n$ is large enough, then
    $\ex_r(n,K_{s,t})=e_r(H'(s-1,t-1,n))$.
\end{proposition}

We highlight two cases where we can give new, much simpler proofs to known theorems using generalized Turán results. Proving a conjecture of Caro and Yuster \cite{cy}, Nikiforov \cite{nik} showed the following.

\begin{thm}[Nikiforov, \cite{nik}]\label{niki}
   If $r\ge 2$, then $\ex_r(n,C_{2k})=(1+o(1))(k-1)n^r$.
\end{thm}

Let us remark here that Nikiforov proved this theorem for every real $r$, not only for integers. Our proof works only for integers. However, Nikiforov showed that the case $r=2$ easily implies the statement for larger values of $r$.  Let $F_n$ denote the friendship graph, which contains a vertex of degree $n-1$ and a largest possible matching on the other vertices.


\begin{thm}[Chang, Chen and Zhang, \cite{ccz}]\label{kin} For any $n$-vertex $C_4$-free graph $G$, if 
$|E(G)|\le |E(F_n)|=\lfloor 3(n-1)/2\rfloor$, then $e_r(G)\le e_r(F_n)$.   
\end{thm}

This strengthens a theorem of Caro and Yuster \cite{cy}, who proved that $\ex_r(n,\{C_4,C_6,\dots\})\le e_r(F_n)$. Indeed, it is well-known and easy to see that a $\{C_4,C_6,\dots\}$-free graph has at most $\lfloor 3(n-1)/2\rfloor$ edges. 
We obtain the following further strengthening.

\begin{thm}\label{stars}
    \textbf{(i)} For any $n$-vertex $C_4$-free graph $G$, we have that $e_r(G)\le e_r(F_n)+|E(G)|-|E(F_n)|$.

    \textbf{(ii)} For $r\ge 3$ and sufficiently large $n$, we have that $\ex_r(n,C_4)=e_r(F_n)$.

    \textbf{(iii)}  For any $k>2$ and sufficiently large $n$, $\ex_2(n,\{C_4,C_{2k}\})=e_2(F_n)$.
\end{thm}

An edge of a graph is \textit{color-critical} if by deleting that edge the chromatic number decreases. Simonovits \cite{S} showed that if $F$ has chromatic number $k+1$ and a color-critical edge, then $\ex(n,F)=\ex(n,K_{k+1})$ for sufficiently large $n$.

Gu, Li and Shi \cite{gls} proved that for sufficiently large $n$ we have $\ex_r(n,C_5)=\ex_r(n,K_3)=e_r(T)$ for some $n$-vertex complete bipartite graph $T$. 
We will greatly extend this result by showing that for any graph $F$ with a color-critical edge, $\ex_r(n,F)=e_r(T)$ for some $n$-vertex complete $(\chi(F)-1)$-partite graph $T$. In fact, we prove even more.

Given a graph $F$ with $\chi(F)=k+1$, its decomposition family $\cD(F)$ is the family of bipartite graphs obtained by taking two classes in a proper $(k+1)$-coloring of $F$. Let $\mathrm{biex}(n,F)=\ex(n,\cD(F))$.

\begin{thm}\label{biexes}
    For any non-bipartite graph $F$ we have that $\ex_r(n,F)=e_r(T)+\Theta(\mathrm{biex}(n,F)n^{r-1})$ for some $n$-vertex complete $(\chi(F)-1)$-partite graph $T$.
\end{thm}

Note that if $F$ has a color-critical edge, then $\mathrm{biex}(n,F)=0$, thus indeed $\ex_r(n,F)=e_r(T)$ in this case.
We will prove a structural theorem (Corollary \ref{cori}) in Section \ref{four} that we will use to prove the above theorem. It also implies 
the following exact results for graphs without a color-critical edge. 

We say that a graph is \textit{almost $\ell$-regular} if either each vertex has degree $\ell$, or one vertex has degree $\ell-1$ and all the other vertices have degree $\ell$.
Let $\cT_0(n,k,a-1)$ be the family of $n$-vertex graphs that can be obtained from a complete $k$-partite graph by adding an almost $(a-1)$-regular graph into each part. Let $B_{k,1}$ denote two copies of $K_k$ sharing exactly one vertex.

\begin{thm}\label{labe}
    \textbf{(i)} Let $F$ consist of $s>1$ components with chromatic number $k+1$, each with a color-critical edge, and any number of components with chromatic number at most $k$. Let $n$ be sufficiently large. Then $\ex_r(n,F)=e_r(T)$ for a complete $(s+k-1)$-partite graph $T$ with $s-1$ parts of order 1.

    \textbf{(ii)} Let $F$ be the complete $(k+1)$-partite graph $K_{1,a,\dots,a}$ and $n$ be sufficiently large. Then $\ex_r(n,F)=e_r(T)$ for some $T\in \cT_0(n,k,a-1)$.

    \textbf{(iii)} We have $\ex_r(n,B_{k+1,1})=e_r(T')$, for some $T'$ that is obtained from a complete $k$-partite graph by adding an edge into one of the parts.
\end{thm}

Note that \textbf{(i)} generalizes a theorem of Moon \cite{moon} on $\ex(n,F)$.



There are multiple results, analogous to the above theorem that extend results on $\ex(n,F)$ to $\ex(n,H,F)$ for a large class of 
graphs $H$. However, there are no results that automatically extend results from $\ex(n,F)$ to $\ex(n,H,F)$. Here we provide one for stars $H$ and also for $\ex_r(n,F)$.

\begin{thm}\label{newmain}
    Let $\mathrm{biex}(n,F)=O(1)$ and assume that for sufficiently large $n$ the extremal graphs for $\ex(n,F)$ each have the following form: we take the Tur\'an graph $T(n,\chi(F)-1)$ and add a $\cD(F)$-free graph with $\mathrm{biex}(n,F)$ edges to one of the parts. Then for sufficiently large $n$ the extremal graphs for $\ex(n,S_r,F)$ and $\ex_r(n,F)$ each have the following form: we take a complete $(\chi(F)-1)$-partite graph and add a $\cD(F)$-free graph with $\mathrm{biex}(n,F)$ edges to one of the parts.
\end{thm}

This theorem determines $\ex_r(n,F)$ (apart from the optimization of the order of the parts) for the following graphs. We take some graphs of chromatic number $k$ with color-critical edges, and then add a new vertex and connect it to each other vertex. This includes $B_{k+1,1}$ (thus proves \textbf{(iii)} of Theorem \ref{labe}) among other graphs and can be proved using a theorem from \cite{hlz}. Another example is obtained by taking $k$ odd cycles sharing exactly one vertex, using \cite{yuan}.

In Section \ref{two} we discuss the connection between degree powers and counting stars and prove our simpler results. We prove Theorem \ref{stars} in Section \ref{three}. We deal with the non-degenerate case, i.e., the case of non-bipartite forbidden graphs in Section \ref{four}.

\section{The connection between degree powers and counting stars}\label{two}

Let us show some cases where results in one area imply results in the other area using Proposition \ref{csill}. Note that generalized Turán results are typically newer, but are more general than the form stated here.

Caro and Yuster \cite{cy} showed that if $1<s\le r$, then $\ex_r(n,K_{s,r})=(1+o(1))(s-1)n^r$, and if $t\le r$, then $\ex_r(n,K_{2,t})=(1+o(1))n^r$. Later, Gerbner and Patk\'os \cite{gpat} showed that if $1<s<r$ or $1<s=t=r$, then $\ex(n,S_r,K_{s,t})=(1+o(1))(s-1)\binom{n}{r}$. The same holds if $1<s\le t<r$, and the extremal graph is also determined in \cite{gpat} for $n$ sufficiently large. These results imply the result of Caro and Yuster, and also that if $1<s\le t<r$, then $\ex_r(n,K_{s,t})=(1+o(1))(s-1)n^r$. Later we will determine $\ex_r(n,K_{s,t})$ exactly for these parameters.

Bollob\'as and Nikiforov \cite{bolnik} showed that if $H$ has chromatic number $k$, then $e_r(n,H)=e_r(n,K_k)+o(n^r)$. This is implied by a theorem of Gerbner and Palmer \cite{gp2}. Bollob\'as and Nikiforov \cite{bolnik} also showed that if $r< k$, then $\ex_r(n,K_{k+1})=e_r(T(n,k))$. This implies \textbf{(i)} of Proposition \ref{kovik}.

Nikiforov \cite{nik} proved Theorem \ref{niki}. Recall that it states that $e_r(n,C_{2k})=(1+o(1))(k-1)n^r$. Surprisingly, we are not aware of any result stated on $\ex(n,S_r,C_{2k})$. However, Gerbner, Nagy and Vizer \cite{gnv} showed that $\ex(n,K_{2,r},C_{2k})=(1+o(1))\binom{k-1}{2}\binom{n}{r}$ (the case $r=2$ was proved in \cite{GGyMV}). 
We show that this implies Nikiforov's result. 

\begin{proof}[New proof of Theorem \ref{niki}]
    Let us consider an arbitrary ordering of the $r$-sets in an $n$-vertex $C_{2k}$-free graph $G$, and let $p_i$ denote the order of the common neighborhood of the $i$th set. Then we have $\sum_{i=1}^{\binom{n}{r}} \binom{p_i}{2}\le \ex(n,K_{2,r},C_{2k})=(1+o(1))\binom{k-1}{2}\binom{n}{r}$, thus $\sum_{i=1}^{\binom{n}{r}} p_i^2\le (1+o(1))(k-1)^2\binom{n}{r}$. Therefore, $\cN(S_r,G)=\sum_{i=1}^{\binom{n}{r}} p_i\le \sqrt{\binom{n}{r}\sum_{i=1}^{\binom{n}{r}} p_i^2}\le (1+o(1))(k-1)\binom{n}{r}$, where the power mean inequality is used. By Proposition \ref{csill}, this implies Theorem \ref{niki}.
\end{proof}


Let us continue with Proposition \ref{exac} that we restate here for convenience.

\begin{proposition*}
    There are positive numbers $w_p=w_p(r)$ such that $e_r(G)=\sum_{p=1}^r w_p\cN(S_p,G)$.
\end{proposition*}

\begin{proof}
Consider a vertex $v$ of $G$ with neighbors $u_1,\dots,u_d$. One can see its contribution $d^r$ to $e_r(G)$ as the number of vectors of length $r$ with entries $u_i$. In the case the entries are each different vertices, this corresponds to a star $S_r$ with center $v$. Each such star is counted $r!$ times. Note that this is the same argument that led to Proposition \ref{csill}. However, consider now vectors of other types. Each vector corresponds to a star $S_p$ with $p\le r$, with center $v$ and leaves $u_i$ that appear as entries in the vector. Then each copy of $S_p$ with center $v$ is counted as many times as the number of ways we can form a vector of length $r$ with the leaves as entries (and each leaf appearing at least once). Observe that this number does not depend on the choice of the $p$ leaves or on $d$, only on $r$ and $p$, completing the proof. 
\end{proof}

Note that the exact value of $w_p$ is not going to be important for us, except that $w_1=1$. It is not hard to see that $w_p=\sum_{i=0}^{p-1} (-1)^{i}\binom{p}{i}(p-i)^{r}$. 
Indeed, $(p-i)^r$ counts the vectors with a given set of $p-i$ entries where some of the entries may appear zero times. Observe that $\binom{p}{i}(p-i)^{r}$ counts the vectors for all the $(p-i)$-sets. If we are given a set $P$ with exactly $p-i$ entries, it is also counted as a vector with at most $p-j$ entries for $j< i$, simply for each $(p-j)$-set that contains $P$, i.e., $\binom{i}{j}$ times. Therefore, $P$ is counted $\sum_{j=0}^i (-1)^{j}\binom{i}{j}$ times. This counts the number of even subsets minus the number of odd subsets of an $i$-element set, thus equals zero unless $i=0$.




The above proposition shows that $e_r(G)$ is a generalized Tur\'an problem with counting multiple graphs at the same time (with weights). Such problems were studied in \cite{gerbn}, but without results relevant to us. One case when it is easy to deal with counting multiple graphs is when the same graph is extremal for them. This is the case for example if $F$ has a color-critical edge and chromatic number at least $300(r+1)^9+1$ \cite{mnnrw,dahi}.
This already proves a special case of Theorem \ref{biexes}.

We have already mentioned that the case $r$ is not an integer has been studied. In fact, for the degree sequence $(d_i)$ of $G$, $\sum_{i=1}^n f(d_i)$ has been studied for more general functions $f$. Bollob\'as and Nikiforov \cite{bolnik2} specifically suggested counting $\sum_{i=1}^n \binom{d_i}{r}$, which is $\cN(S_r,G)$. Li and Shi \cite{lish} studied this. In our language, they proved that if $r\le 2$, then $\ex(n,S_r,K_{k+1})=\cN(S_r,T(n,k))$. The case $r=1$ is Tur\'an's theorem, the case $r=2$ was known at that point \cite{gypl} for $k=2$, but not for larger $k$. They showed that for any $r$, if $k$ is large enough, then  $\ex(n,S_r,K_{k+1})=\cN(S_r,T(n,k))$. This was proved later for complete multipartite graphs $H$ in place of $S_r$ in \cite{gerpal} and for any graph $H$ in \cite{mnnrw}. They also showed that for any graph $F$ with chromatic number $k+1$, we have $\ex(n,S_r,F)=\ex(n,S_r,K_{k+1})+o(n^{r+1})$. This was later proved for every graph $H$ in \cite{gp2}.

Proposition \ref{exac} is also useful when we have a stability result. 
Gerbner and Patk\'os \cite{gpat} showed that if $1<s\le t<r$ and $n$ is sufficiently large, then any $n$-vertex $K_{s,t}$-free graph $G$ with at least $\ex(n,S_r,K_{s,t})-\Omega(n^{r-1})$ copies of $S_r$ contains $H(s-1,n)$. Moreover, $\ex(n,S_r,K_{s,t})=\cN(S_r,H'(s-1,t-1,n))$. Using this, we are going to prove \textbf{(ii)} of Proposition \ref{kovik}, which states that
if $1<s\le t<r$ and $n$ is large enough, then
    $\ex_r(n,K_{s,t})=e_r(H'(s-1,t-1,n))$.

\begin{proof}[Proof of \textbf{(ii)} of Proposition \ref{kovik}]
For the lower bound, assume indirectly that $H'(s-1,t-1,n)$ contains a copy of $K_{s,t}$. That copy contains at least $t+1$ vertices of degree less than $n-1$. Those vertices form a copy of $K_{a,b}$ inside $G_0$ for some $a\le b$ with $a+b\ge t+1$. Observe that if $a=1$, we have a $K_{1,t}$ in $G_0$, while if $a>1$, then $G_0$ contains $K_{a,a}$, which contain $C_4$. Both are impossible by the construction of $G_0$.

Let us continue with the upper bound and let $G$ be an extremal graph. Assume first that $G$ does not contain $H(s-1,n)$. Then by the result of Gerbner and Patk\'os \cite{gpat} mentioned before the proof, $G$ contains at most $\ex(n,S_r,K_{s,t})-\Omega(n^{r-1})$ copies of $S_r$. We also have that $G$ contains at most $\ex(n,S_{r-1},K_{s,t})=(1+o(1))\cN(S_{r-1},H(s-1,n))$ copies of $S_{r-1}$, using that $r-1\ge t$ and the results on $\ex(n,S_r,K_{s,t})$ mentioned earlier.

Now we apply Proposition \ref{exac}. We have that 
\begin{align*}
    e_r(G)=\sum_{i=1}^r w_p\cN(S_p,G)\le O(n^{r-2})+w_{r-1}\cN(S_{r-1},G)+w_r\ex(n,S_r,K_{s,t})-\Omega(n^{r-1})\le \\ w_{r-1}\cN(S_{r-1},H(s-1,n))+o(n^{r-1})+w_r\cN(S_r,H(s-1,n))+O(n)-\Omega(n^{r-1})=\\ w_{r-1}\cN(S_{r-1},H(s-1,n))+w_r\cN(S_r,H(s-1,n))-\Omega(n^{r-1})<e_r(H(s-1,n)),
\end{align*}
a contradiction.

We obtained that $G$ contains $H(s-1,n)$. Let $A$ denote the set of vertices of degree less than $n-1$ in this $H(s-1,n)$. Observe that each vertex in $A$ can have at most $t-1$ neighbors in $A$ because $G$ is $K_{s,t}$-free. Therefore, each degree inside $A$ is at most $t-1$, and if $(t-1)(n-s+1)$ is odd, then at least one of the degrees is at most $t-2$. This shows that the for every $i$, the $i$th largest degree in $G$ is at most the $i$th largest degree in $H'(s-1,t-1,n)$, completing the proof. 
\end{proof}

\section{Forbidden $C_4$}\label{three}

Proposition \ref{exac} can also be useful when the same graph is extremal only for almost all the stars $S_p$ with $p\le r$. Consider now $F=C_4$. 

\begin{proposition}[Gerbner, \cite{ger}]\label{small} $\ex(n,S_r,C_4)=\cN(S_r,F_n)$ for $r\ge 2$.    
\end{proposition}

In the case $r=1$, we know that $\ex(n,C_4)=(1+o(1))n^{3/2}/2$ \cite{fure}, much larger than $|E(F_n)|$. That means that for an $n$-vertex $C_4$-free graph $G$ we may have that $e_r(G)>e_r(F_n)$ for larger $r$, but that surplus must come from the number of edges. This proves \textbf{(i)} of Theorem \ref{stars}. We remark that the proofs of the weaker statements in \cite{cy} and \cite{ccz} are much longer than this proof, even if we add the proof of Proposition \ref{small} from \cite{ger} to this proof.


To prove \textbf{(ii)} and \textbf{(iii)} of Theorem \ref{stars}, we need a stability version of Proposition \ref{small}.

\begin{lemma}\label{lemmike}
    If an $n$-vertex $C_4$-free graph $G$ does not contain $S_{n-1}$, then $\cN(S_r,G)\le \binom{n-1}{r}-\Omega(n^{r-1})$. Furthermore, 
    for any constant $c$ there is a constant $c'$ such that if the largest degree in $G$ is less than $n-c'\sqrt{n}$, then $\cN(S_r,G)\le \binom{n-1}{r}-cn^{r-1/2}$.
\end{lemma}

\begin{proof} Let $\Delta$ be the largest degree in $G$.
    We follow the proof of Proposition \ref{small} from \cite{ger}. We count the copies of $S_r$ by picking two leaves first ($\binom{n}{2}$ ways), then their common neighbor (at most 1 way), and $r-2$ other neighbors of that vertex (at most $\binom{\Delta-2}{r-2}$ ways). Each copy of $S_r$ is counted $\binom{r}{2}$ ways. Therefore, we have $\cN(S_r,G)\le \binom{n}{2}\binom{\Delta-2}{r-2}/\binom{r}{2}\le \binom{n}{2}\binom{n-c'\sqrt{n}}{r-2}/\binom{r}{2}$, where the last inequality holds in the case $\Delta\le n-c'\sqrt{n}$. The main term of this upper bound is $n^r/r!$. Then we have the terms where the $c'\sqrt{n}$ appears exactly once, and $n$ appears $r-1$ times. This gives $-\Theta(c'n^{r-1/2})$, and each other term is $O(n^{r-1})$.

    If $\Delta\ge n-c'\sqrt{n}$, let $u$ be a  vertex of degree $\Delta$, 
    $U$ be the set of neighbors of $u$ and $U'$ be the set of other vertices. Assume that $|U'|$ is large enough.
    Observe that there are at most $\ex(|U'|,S_r,C_4)=\binom{|U'|-1}{r}$ copies of $S_r$ inside $U'$ (using Proposition \ref{small}),
thus there is a vertex $v$ in $U'$ that is contained in at most $(r+1)\binom{|U'|-1}{r}/|U'|$ copies of $S_r$ inside $U'$. There is at most one neighbor $v'$ of $v$ in $U$. As $v'$ also has at most one neighbor in $U$, the degree of $v'$ is at most $|U'|+2$. Therefore, we have that $vv'$ is in at most $\binom{|U'|-1}{r-1}+\binom{|U'|+1}{r-1}$ copies of $S_r$. Altogether, $v$ is in $O(|U'|^{r-1})=o(n^{r-1})$ copies of $S_r$. Let us delete the edges incident to $v$ and add the edge $uv$ to obtain $G'$. It is easy to see that no $C_4$ is created this way, and the number of copies of $S_r$ increases by $\Omega(n^{r-1})$, thus we have $\cN(S_r,G)\le \cN(S_r,G')-\Omega(n^{r-1})\le\ex(n,S_r,G)-\Omega(n^{r-1})=\binom{n-1}{r}-\Omega(n^{r-1})$. 

If $|U'|$ is not large enough, then each vertex except for $u$ has degree $O(1)$, thus $v$ is in $O(1)$ copies of $S_r$ and again $\cN(S_r,G)\le \cN(S_r,G')-\Omega(n^{r-1})$. This also shows that if $r\ge 3$, then in the extremal graph $U'$ must be empty, hence $\Delta=n-1$ and it is easy to see that each other vertex has degree less than $r$ because of the $C_4$-free property.
\end{proof}

We also need the following result.

\begin{proposition}\label{korok}
    If $\ell\ge 6$, then $\ex(n,C_\ell,\{C_4,C_{2k}\})=O(n^{\lfloor\ell/3\rfloor})$.
    
\end{proposition}

In the case $\ell\neq 2k\neq 6$, the above bound is sharp, as shown by taking vertices $u_1,\dots,u_{\lfloor\ell/3\rfloor}$, adding linearly many 3-edge paths between $u_i$ and $u_{i+1}$ for each $i$, and adding an edge, a 2-edge path or linearly many 3-edge paths between $u_1$ and $u_{\lfloor\ell/3\rfloor}$.

Note that counting cycles when forbidding a set of cycles was studied in \cite{GGyMV}, and the above statement was proved in the case $\ell$ is divisible by 3. It is not hard to extend the proof in \cite{GGyMV} to our case. First, we gather the ingredients from \cite{GGyMV} in the following proposition. Let $f(u,v)$ denote the number of paths of 3 edges between $u$ and $v$.

\begin{proposition}\label{korrok}
    \textbf{(i)} $\ex(n,C_\ell,\{C_4,C_{2k}\})=\Theta(\ex(n,C_\ell,\{C_3,C_4,C_{2k}\}))$.

    \textbf{(ii)} Let $G$ be an $n$-vertex $\{C_3,C_4,C_{2k}\}$-free graph. For any $u\in V(G)$ we have $\sum_{v\in V(G)} f(u,v)=O(n)$. 

    \textbf{(iii)} Let $G$ be an $n$-vertex $\{C_3,C_4,C_{2k}\}$-free graph. We have $\sum_{u,v\in V(G)} f^2(u,v)=O(n^2)$.

    \textbf{(iv)} If $\ell\ge 6$ and $\ell$ is divisible by 3, then $\ex(n,C_\ell,\{C_4,C_{2k}\})=O(n^{\ell/3})$.
\end{proposition}

Note that \textbf{(i)} follows from Lemmas 5.1 and 5.2, \textbf{(ii)} follows from Claim 17 and \textbf{(iii)} follows from (5.1) and \textbf{(iv)} is a special case of Theorem 2.6 in \cite{GGyMV}.
We remark that in addition to using the above results, the rest of the proof also closely follows the proof of Theorem 2.6 from \cite{GGyMV}.

\begin{proof}[Proof of Proposition \ref{korok}] Assume that $\ell$ is not divisible by 3 (otherwise \textbf{(iv)} of Proposition \ref{korrok} completes the proof) and let $m=\lceil \ell/3\rceil$. Let $G$ be an $n$-vertex $\{C_3,C_4,C_{2k}\}$-free graph and fix $v_1,\dots,v_m$. Let us count the copies of $C_\ell$ where $v_i$ is the $(3i-2)$nd vertex. Then $v_m$ is the $\ell$'th vertex or the $\ell-1$'st vertex, in the second case there is at most one way to pick the $\ell$'th vertex. Thus to complete the cycle, we need to pick paths of length 3 between the consecutive vertices, hence there are at most $\prod_{i=1}^{m-1}f(v_i,v_{i+1})$ such cycles.
We obtained that

\[\cN(C_\ell,G)\le\sum_{v_1,\dots,v_m\in V(G)}\prod_{i=1}^{m-1}f(v_i,v_{i+1})\le\sum_{v_1,\dots,v_m\in V(G)}\frac{f(v_1,v_2)^2+f(v_2,v_3)^2}{2}\prod_{i=3}^{m-1}f(v_i,v_{i+1}).\]

Fix two vertices $u$ and $v$ and let us examine what factor $f^2(u,v)$ is multiplied with. It appears in the right hand side of the above inequality if $\{u,v\}$ is either $\{v_1,v_2\}$ or $\{v_2,v_3\}$. This gives four possibilities, we examine the case $u=v_1$ and $v=v_2$, the other three cases work the same way. In our case, $f^2(u,v)$ is multiplied with $\prod_{i=3}^{m-1}f(v_i,v_{i+1})=O(n^{m-3})$ using \textbf{(ii)} of Proposition \ref{korrok}. Therefore, $\cN(C_\ell,G)=O(\sum_{u,v\in V(G)}f^2(u,v)n^{m-3})=O(n^{m-1})$, using \textbf{(iii)} of Proposition \ref{korrok} and completing the proof.
\end{proof}

Now we are ready to prove the rest of Theorem \ref{stars} that we restate here for convenience.

\begin{theorem*}
    \textbf{(ii)} For $r\ge 3$ and sufficiently large $n$, we have that $\ex_r(n,C_4)=e_r(F_n)$.

    \textbf{(iii)}  For any $k>2$ and sufficiently large $n$, $\ex_2(n,\{C_4,C_{2k}\})=e_2(F_n)$.
\end{theorem*}

\begin{proof}
To prove \textbf{(ii)}, let $G$ be an $n$-vertex $C_4$-free graph. If $G$ contains $S_{n-1}$, then clearly only independent edges can be added, $G$ is a subgraph of $F_n$ and we are done. Otherwise $\cN(S_r,G)<\cN(S_r,F_n)-\Omega(n^2)$ by Lemma \ref{lemmike}, thus $\sum_{p=1}^r w_p\cN(S_p,G)\le \sum_{p=1}^r w_p\cN(S_p,F_n)-\Omega(n^2)+|E(G)|<\sum_{p=1}^r w_p\cN(S_p,F_n)$, where we use that $|E(G)|\le \ex(n,C_4)=o(n^2)$.

Let us continue with \textbf{(iii)}. Let $G$ be an $n$-vertex $\{C_4,C_{2k}\}$-free graph. Let $G'$ be the following auxiliary graph. We let $V(G')=V(G)$ and $uv$ is an edge of $G'$ if $u$ and $v$ have a common neighbor in $G$. Then clearly $\cN(S_2,G)=|E(G')|$. Let us consider a copy $C$ of $C_k$ in $G'$. The $k$ edges correspond to $k$ vertices $v_1,\dots,v_k$ of $G$ that are the common neighbors in $G$ of the endpoints of the edge. If $v_1,\dots,v_k$ are distinct from the vertices of $C$ and from each other, we obtain a $C_{2k}$ in $G$, a contradiction. 

There are copies of $C_k$ in $G'$ where some $v_i$ is one of the vertices of $C_k$. We can find each such $C_k$ in $G'$ by picking $k-1$ vertices $O(n^{k-1})$ ways, then picking a pair of those vertices $\binom{k-1}{2}$ ways, then adding the unique common neighbor of the two vertices picked as the $k$th vertex and ordering the $k$ vertices cyclically, $O(1)$ ways.

Assume now that the $v_i$'s are not in $C$. Then there exist $i\neq j$ such that $v_i=v_j$. Then this vertex is adjacent to at least three vertices of $C$. There are two possibilities. Either all the vertices of $C$ are adjacent to $v_i$, clearly there are $(k-1)!\cN(S_k,G)/2$ copies of such cycles, since we can order the $k$ leaves of an $S_k$ cyclically $(k-1)!/2$ ways. Finally, if $v_i\neq v_m$ for some $m$, then there is a cycle in the subgraph $G$ induced by the vertices of $C$ and the common neighbors. Indeed, there are at least two walks from $v_i$ to $v_m$ in $G$, starting with either neighbor of $v_i$ in $C$. Let us pick a shortest such cycle and denote its length by $\ell$. Observe that vertices of $C$ and the common neighbors alternate here, thus $\ell$ is even and there are $\ell/2$ vertices of $C$ in the cycle.
Proposition \ref{korok} implies that $\cN(C_\ell,G)=O(n^{\ell/2-1})$. 
This means that we can count the $k$-cycles in $G'$ of this type by picking an $\ell<2k$ constant many ways, an $\ell$-cycle $O(n^{\ell/2-1})$ many ways, $k-\ell/2$ further vertices $O(n^{k-\ell/2})$ many ways, and order these vertices $O(1)$ ways. Therefore, there are $O(n^{k-1})$ $k$-cycles of this type in $G'$.



We consider two cases. 
Assume first that the largest degree $\Delta$ in $G$ is less than $n-c'\sqrt{n}$, where $c'$ is chosen such that Lemma \ref{lemmike} implies that $\cN(S_k,G)\le \binom{n-1}{k}-2n^{k-1/2}/7k!$.
By the above, this implies that there are at most $(k-1)!\binom{n-1}{k}/2-n^{k-1/2}/7k+O(n^{k-1})$ copies of $C_k$ in $G'$. 

We claim that $G'$ has at most $\binom{n}{2}-4n^{3/2}/15$ edges. 
Indeed, we can extend each missing edge $v_1v_2$ by picking $k-2$ further vertices $v_3,\dots, v_k$ to obtain a missing $k$-cycle with vertices in the order $v_1v_2\dots v_k$, at most $(n-2)(n-3)\dots (n-k+1)$ ways, and each missing $C_k$ is counted at most $2k$ times. This shows that there are at most $2n^{k-1/2}/15k$ missing $k$-cycles in $G'$, a contradiction showing that $\cN(S_2,G)\le\binom{n}{2}-4n^{3/2}/15$. Therefore, we have $e_2(G)=2\cN(S_2,G)+|E(G)|<(n-1)^2<e_2(F_n)$ (using that $\ex(n,C_4)\le (1+o(1))n^{3/2}/2$), completing the proof in this case.

Assume now that there is a vertex $u$ in $G$ of degree at least $n-c'\sqrt{n}$. Let $U$ be the neighborhood of $u$ and $U'$ be the rest of the vertices. Then there are at most $O(|U'|^2)$ copies of $S_2$ inside $U'$ and at most $O(|U'|^{3/2})$ edges inside $U'$. There are at most $|U|/2$ edges inside $U$. There are at most $|U'|$ edges between $U$ and $U'$, let $xy$ be one of them. The endpoint $x$ in $U$ is incident only to these edges and one more edge inside $U$, and the edge $ux$. The other endpoint $y$ is adjacent only to $x$ and vertices in $U'$. Therefore, $xy$ is in at most $2|U'|+1$ copies of $S_2$, thus $O(|U'|^{2})$ copies of $S_2$ contain at least one of the edges between $U$ and $U'$. As there are no copies of $S_2$ inside $U$ we obtain that there are $O(|U'|^{2})$ copies of $S_2$ in $G$ that do not contain $u$. 

Let us now delete each edge incident to $U'$ and connect each vertex of $U'$ to $u$. This way we deleted $O(|U'|^2)=O(|U'|\sqrt{n})$ copies of $S_2$ and $O(|U'|^2)=O(|U'|\sqrt{n})$ edges, and added $\Omega(n|U'|)$ copies of $S_2$, thus if $U'\neq\emptyset$, then $e_2$ increases, a contradiction. If $U'=\emptyset$, then it is easy to see that we can only add independent edges to $S_{n-1}$, thus $G$ is a subgraph of $F_n$, completing the proof.
 \end{proof}

\section{The non-degenerate case}\label{four}

We say that a graph $H$ is weakly $F$-Tur\'an-stable if any $n$-vertex $F$-free graph with at least $\ex(n,H,F)-o(n^{|V(H)|})$ copies of $H$ can be turned into a complete $(\chi(F)-1)$-partite graph by removing and adding $o(n^2)$ edges. The first result concerning such graphs is due to Ma and Qiu \cite{mq}, who showed that $K_\ell$ is weakly $F$-Tur\'an-stable for any $F$ with $\chi(F)>\ell$, and used this to prove that $\ex(n,K_\ell,F)=\cN(K_\ell,T(n,\chi(F)-1))+\Theta(\mathrm{biex}(n,F)n^{\ell-2})$. Gerbner \cite{ger2,ger3} showed that this can be generalized to other graphs in place of $K_\ell$. Gerbner and Hama Karim \cite{dahi} showed that complete $\ell$-partite graphs are weakly $F$-Tur\'an-stable for any $F$ with $\chi(F)>\ell$. This implies that stars are weakly $F$-Tur\'an-stable for any non-bipartite graph $F$. 

Observe that a stability property immediately applies to our setting. Assume that an $n$-vertex $F$-free graph $G$ has $e_r(G)\ge\ex_r(n,F)-o(n^{r+1})$. Then we have that $\cN(S_r,G)=e_r(G)/r!+O(n^r)=\ex_r(n,F)/r!-o(n^{r+1})=\ex(n,S_r,F)-o(n^{r+1})$. Therefore, using that $S_r$ is weakly $F$-Tur\'an-stable, we have that $G$ can be turned into a complete $(\chi(F)-1)$-partite graph by removing and adding $o(n^2)$ edges. 

However, the next step in the proofs of analogous results in the generalized Tur\'an setting is to prove a structural theorem on an extremal $n$-vertex graph $G$. That step uses that $G$ contains $\ex(n,H,F)$ copies of $H$. In our setting the extremal graph $G$ has $e_r(G)=\ex_r(n,F)$, but on the number of stars we can only state that $\cN(S_r,G)\ge\ex(n,S_r,F)-O(n^r)$. We cannot use the structural theorem inside a proof in \cite{mq} or the structural theorem in \cite{ger3}. 
Gerbner \cite{ger4} proved a variant of the structural theorem for $\ex(n,K_\ell,F)$, where the assumption is that $\cN(K_\ell,G)\ge \ex(n,K_\ell,F)-O(n^{\ell-1})$. Here we extend this to arbitrary $H$ with $\chi(H)=\ell$.

We will use the following simple lemma from \cite{ger2}.

\begin{lemma}\label{neww}
Let us assume that $\chi(H)<\chi(F)$ and $H$ is weakly $F$-Turán-stable, thus $\ex(n,H,F)=\cN(H,T)+o(n^{|V(H)|})$ for some complete $(\chi(F)-1)$-partite $n$-vertex graph $T$. Then every part of $T$ has order $\Omega(n)$.
\end{lemma}

Let $\sigma(F)$ denote the smallest possible order of a color class in a proper $\chi(F)$-coloring of $F$.

\begin{thm}\label{expa} Let $\chi(F)=k+1>\chi(H)$, $H$ be weakly $F$-Tur\'an-stable, $\zeta>0$, $n$ sufficiently large and $G$ be an $n$-vertex $F$-free graph with $\cN(H,G) \ge \ex(n,H, F)-\zeta n^{|V(H)|-1}$. Then there is a $k$-partition of $V(G)$ to $V_1,\dots,V_k$, a constant $K=K(F,\zeta)$, a set $B$ of at most $K(\sigma(F)-1)$ vertices and a set $U$ of at most $K$ vertices such that each member of $\cD(F)$ inside a part shares at least two vertices with $B$, every vertex of $B$ is adjacent to $\Omega(n)$ vertices in each part, every vertex of $U\cap V_i$ is adjacent to $o(n)$ vertices in $V_i$ and every vertex of $V_i\setminus (B\cup U)$ is adjacent to $o(n)$ vertices in $V_i$ and all but $o(n)$ vertices in $V_j$ with $j\neq i$. Moreover, if $\zeta$ is sufficiently small, then $U=\emptyset$ for sufficiently large $n$.\end{thm}

The proof is almost identical to the proofs in \cite{mq,ger3}, thus we only give a sketch. In fact, this sketch is almost identical to the sketch in \cite{ger4}.

\begin{proof}[Sketch of proof] Let us pick numbers $\alpha,\beta,\gamma,\varepsilon>0$ in this order, such that each is
sufficiently small compared to the previous one, and after that we pick $n$ that is sufficiently
large. By the weak Turán-stability, $G$ is close to the some complete $k$-partite graph. More precisely, we can partition the vertex set into $V_1,\dots,V_k$ such that at most $\varepsilon n^2$ edges are missing between parts and $|V_i|\ge \alpha n$ for each $i$ by Lemma \ref{neww}.
Let $B_i$
denote the set of vertices in $V_i$ with at least $\gamma n$ neighbors in $V_i$ and $B:=\cup_{i=1}^k B_i$.

The proof of the upper bound on $|B|$ in \cite{ger3} (and in \cite{mq}) does not use the assumption that $\cN(H,G)=\ex(n,H,F)$, thus it holds in our setting. Let $U_i$ denote the set of vertices in $V_i$ with at least $\beta n$ non-neighbors in some other part $V_j$ and $U:=\cup_{i=1}^k U_i$. In the case $G$ is extremal, it was shown in \cite{mq,ger3} that $U$ is empty, by showing that we can change $G$ to another $F$-free graph $G'$ with $\cN(H,G')\ge\cN(H,G)+|U|\beta n^{|V(H)|-1}/2^{|V(H)|}$. For us, it just shows that $|U|$ is bounded above by a constant $K$, since $\cN(H,G')\le \ex(n,H,G)\le \cN(H,G)+\zeta n^{|V(H)|-1}$. If $\zeta$ is sufficiently small compared to $\beta$, then this implies $|U|=0$.

Finally, the assumption on $\cD(F)$ follows as in \cite{ger3}.
\end{proof}

We remark that in \cite{ger4} the special case $H=K_\ell$ was needed also because we were interested in graphs that were extremal for a related problem. Using that, it was not hard to show that $U$ is empty. Similarly, here we can also show that $U$ must be empty. 

\begin{corollary}\label{cori}
    Let $\chi(F)=k+1>2$, $n$ sufficiently large and $G$ be an $n$-vertex $F$-free graph with $e_r(G)=\ex_r(n,F)$. Then there is a $k$-partition of $V(G)$ to $V_1,\dots,V_k$, a constant $K=K(F)$, a set $B$ of at most $K(\sigma(F)-1)$ vertices such that each member of $\cD(F)$ inside a part shares at least two vertices with $B$, every vertex of $B$ is adjacent to $\Omega(n)$ vertices in each part and every vertex of $V_i\setminus B$ is adjacent to $o(n)$ vertices in $V_i$ and all but $o(n)$ vertices in $V_j$ with $j\neq i$.
\end{corollary}

\begin{proof}[Proof of Corollary \ref{cori}]
We apply Theorem \ref{expa}. Observe that we only need to show that $U=\emptyset$. 
Assume that $u\in U_i$ and pick a set $A$ of $|V(F)|$ vertices from $V_i\setminus (B\cup U)$. Observe that the common neighborhood of the vertices in $A$ contains all but $o(n)$ vertices of every $V_j$ with $j\neq i$.

Now we delete all the edges incident to $u$ and add all the edges from $u$ to each vertex in the common neighborhood of $A$. We claim that the resulting graph $G'$ is $F$-free. Indeed, a copy of $F$ has to contain a new edge, and thus has to contain $u$. But $u$ could be replaced by any vertex in $A$ not in this copy of $F$, to obtain another copy of $F$ in $G'$ that is also in $G$, a contradiction.

Observe that we actually removed $o(n)$ edges, thus removed $o(n^{r})$ copies of $S_r$ (since each edge is in $O(n^{r-1})$ copies of $S_r$). On the other hand, there are $\Omega(n)$ new neighbors of $u$, which creates $\Omega(n^r)$ new copies of $S_r$. This means that $e_r(G)$ also decreases by $o(n^{r})$ and increases by $\Omega(n^r)$, thus $e_r(G')>e_r(G)=\ex_r(n,F)$, a contradiction. 
\end{proof}

Now we are ready to prove Theorem \ref{biexes}.

\begin{proof}
    We apply Corollary \ref{cori}. Let $\chi(F)=k+1$ and $T$ be the complete $k$-partite graph with parts $V_1,\dots,V_k$. Let us prove the upper bound. We will deal with the number of stars in $G$ instead of $e_r(G)$; by Proposition \ref{exac} it is enough to show that the stars that are not in $T$ give $O(\mathrm{biex}(n,F)n^{r-1})$ to the sum. 
    
    The stars outside $T$ each contain an edge inside a part $V_i$. Let us consider first the stars that contain an edge inside a part but no vertex from $B$. Observe that $V_i\setminus B$ is $\cD(F)$-free, hence there are at most $\mathrm{biex}(n,F)$ edges inside $V_i\setminus B$. This implies that there are at most $k\mathrm{biex}(n,F)n^{r-1}$ such stars with at most $r$ leaves. The rest of the stars each contain at least one of the at most $K(\sigma(F)-1)$ vertices of $B$ and at most $r$ other vertices. If $\sigma(F)=1$, then this is 0 and we are done. If $\sigma(F)>1$, then $\mathrm{biex}(n,F)>n-1$, since the star is $\cD(F)$-free. Therefore, $K(\sigma(F)-1)n^r=O(\mathrm{biex}(n,F)n^{r-1})$ and we are done.

    For the lower bound, consider $T$ and place a $\cD(F)$-free graph with $\mathrm{biex}(|V_1|,F)$ edges into $V_1$. It is easy to see that the resulting graph is $F$-free and contains $\Theta(\mathrm{biex}(n,F)n^{r-1})$ copies of $S_r$ that are not in $T$. The contribution of those stars to the sum in Proposition \ref{exac} is $\Theta(\mathrm{biex}(n,F)n^{r-1})$, completing the proof.
\end{proof}

Let us continue with a sketch of the proof of Theorem \ref{labe}.
For each of the statements, it was shown in \cite{ger3} that the same graphs are extremal when the same graph $F$ is forbidden and we count weakly $F$-Turán-stable graphs. If we look at the proofs in \cite{ger3}, we can see that each goes the following natural way: we apply the structural theorem to obtain $k$ parts and then we deal with the edges inside parts. This means that if we already have a $k$-partition given, then the largest number of copies of $H$ can be obtained if the edges inside the parts are added analogously ($s-1$ vertices inside some parts are connected to everything in \textbf{(i)}, adding almost $(a-1)$-regular graphs into each part in \textbf{(ii)} and adding a single edge into one of the parts in \textbf{(iii)}).

Consider now our problem. We apply Corollary \ref{cori}, and we obtain $k$ parts. After that, by the above argument, the same set of edges inside parts creates the most additional copies of any weakly $F$-Turán-stable graph $H$. In particular, the same set of edges creates the most copies of each of $S_1,S_2,\dots,S_r$. Note that this does not necessarily hold for every $F$; this is a property of the constructions presented here that they do not depend on $H$. By creating the most copies of the stars, the same set of edges comes with the largest $\sum_{p=1}^r w_p\cN(S_p,G)$, which is equal to $e_r(G)$ by Proposition \ref{exac}.

Finally, we present the proof of Theorem \ref{newmain} that we restate here for convenience.

\begin{theorem*}
    Let $\mathrm{biex}(n,F)=O(1)$ and assume that for sufficiently large $n$ the extremal graphs for $\ex(n,F)$ each have the following form: we take a Tur\'an graph $T(n,\chi(F)-1)$ and add a $\cD(F)$-free graph with $\mathrm{biex}(n,F)$ edges to one of the parts. Then for sufficiently large $n$ the extremal graphs for $\ex(n,S_r,F)$ and $\ex_r(n,F)$ each have the following form: we take a complete $(\chi(F)-1)$-partite graph and add a $\cD(F)$-free graph with $\mathrm{biex}(n,F)$ edges to one of the parts.
\end{theorem*}


\begin{proof}
    Let $G$ be the extremal graph for $\ex(n,S_r,F)$ or $\ex_r(n,F)$, i.e., an $n$-vertex $F$-free graph satisfying either $\cN(S_r,G)=\ex(n,S_r,F)$ or $e_r(G)=\ex_r(n,F)$. We remark that the extremal graphs may be different from these two problems, but we handle them together. Let us apply Corollary \ref{cori} to bound $\ex_r(n,F)$ and the analogous result from \cite{ger3} to bound $\ex(n,S_r,F)$. In both cases we obtain a $(\chi(F)-1)$-partition to $V_1,\dots,V_k$ of the vertices and a set $B$. Note that the orders of these parts may also be different for the two problems, but we handle them together. Let $\chi(F)=k+1$ and $\mathrm{biex}(n,F)=K_0$. 
    
    We claim that for each $i$, there are at most $K_0$ edges inside $V_i$. Indeed, otherwise, there is a graph $F_0\in \cD(F)$ in $G[V_i]$. Let $F_1$ be the graph obtained by adding a complete $(k-1)$-partite graph $K_{t,\dot,t}$ to $F_0$ and joining each vertex of this complete multipartite graph to each vertex of $F_0$. Then $F_1$ contains $F$ by the definition of $\cD(F)$. Consider now the copy of $F_0$ inside $V_i$. 
    The vertices of $F_0$ are joined to all but $o(n)$ vertices of $V_j$, thus we can pick $t$ common neighbors. We go through the sets $V_j$, and each time we pick $t$ common neighbors of the $O(1)$ vertices picked earlier. This is doable since each of those vertices is adjacent to all but $o(n)$ vertices of $V_j$. At the end, we obtain a copy of $F_1$ in $G$, a contradiction.
    
    Let $T$ be the complete $k$-partite graph with parts $V_1,\dots,V_k$ and assume that $|V_1|\le |V_2|\le\dots |V_k|$. 
    Each edge of $G$ is between vertices in $V_i$ and $V_j$ for some $i$ and $j$ that may be equal. Observe that for two edges between vertices in $V_i$ and $V_j$, the number of copies of $S_r$ containing them is the same apart from a $o(n^{r-1})$ term, since the neighborhoods of the vertices of those edges in $V_i$ and in $V_j$ are the same apart from $o(n)$ vertices. Moreover, the same holds for non-edges between $V_i$ and $V_j$.
    

       We will order the sets $\{i,j\}$, where $i$ may be equal to $j$. Let \[f(i)=\binom{\sum_{j\neq i}|V_j|}{r-1},\] then $f(i)+f(j)$ is roughly the number of stars with an edge between $V_i$ and $V_j$ (apart from a $o(n^{r-1})$ term). We let $\{i,j\}$ come before $\{\ell,m\}$ if $f(i)+f(j)>f(\ell)+f(m)$, in the case of equality we order arbitrarily. For the set $\{i,i\}$, we have $f(i)+f(i)$ in the formula deciding the order.

Consider an initial segment of this ordering and let $m$ be the largest integer with $\{m\}$ being in the initial segment. Then each $\{i,j\}$ with $i,j\le m$ is in this segment. Let $U=V_1\cup\dots\cup V_m$.

\begin{clm}\label{aob}
    $G[U]$ contains at most $|E(T[U])|+K_0$ edges. Moreover, if there are more than $K_0$ edges inside parts, then $G[U]$ contains at most $|E(T[U])|+K_0-1$ edges.
\end{clm}
\begin{proof}[Proof of Claim]
    Assume otherwise and let $K_0+K_1$ be the number of edges inside parts in $G[U]$, thus $K_1> 0$. Let $V_i'$ be the set of vertices in $V_i$ incident to edges inside $V_i$. Recall that there are at most $K_0$ edges inside $V_i$, thus $|V_i'|\le 2K_0$. Also recall that for $n\ge n_0$ we have that $\ex(n,F)=|E(T(n,k))|+K_0$. We will extend $V_i'$ to a set $V_i''$ of order $\max\{2K_0,n_0\}$ the following way. We go through the integers from 1 to $m$. For each $i$, we select $\max\{2K_0,n_0\}-|V_i'|$ vertices from $V_i$ that are in the common neighborhood of at most $\max\{2(m-1)K_0,(m-1)n_0\}$ already picked vertices in $V_j'$ or $V_j''$ for some $j\neq i$. This is doable, since those $O(1)$ vertices each have $o(n)$ non-neighbors in $V_i$, while $|V_i|=\Theta(n)$ by Lemma \ref{neww}. Let $H$ denote the resulting subgraph of $G[U]$.

    In $G[U]$, the number of edges inside parts is more than the number of missing edges between parts by at least $K_0+1$ because of our indirect assumption. The same obviously holds for $H$. Let us extend $H$ by $\max\{n_0,2K_0\}$ vertices for $V_{m+1},\dots,V_k$ the same way, by picking common neighbors of the already picked vertices. Then the resulting graph $H'$ is a balanced $k$-partite graph on $k\max\{n_0,2K_0\}$ vertices, and the number of missing edges between parts did not increase, hence it is less than $K_1$. Clearly $H'$ contains at most $\ex(k\max\{n_0,2K_0\},F)=|E(T(k\max\{n_0,2K_0\},k))|+K_0$ edges, hence the number of missing edges between parts is at least $K_1$, a contradiction. To see the moreover part, observe that in the case of equality, we must have $\ex(k\max\{n_0,2K_0\},F)=|E(T(k\max\{n_0,2K_0\},k))|+K_0$, thus there are exactly $K_0$ edges inside parts.
\end{proof}  

Let us return to the proof of the theorem. Assume first that $G$ contains $K_0+K_2$ edges inside parts for some $K_2>0$. Then 
$K_2\le (k-1)K_0$ and by Claim \ref{aob}, there are at least $K_2+1$ edges missing between parts. We give an 
ordering of the edges inside parts and the missing edges between parts according to the parts that contain the endpoints: an edge between $V_i$ and $V_j$ comes before an edge between $V_\ell$ and $V_m$ if $\{i,j\}$ comes before $\{\ell,m\}$. If $\{i,j\}=\{\ell,m\}$, we let the edges be incomparable, thus incomparability is an equivalence relation. 

Let $e_1,\dots, e_{K_0+K_2}$ be the edges inside parts according to the ordering, and $e_1',\dots,e_{K_2+1}'$ be the first $K_2+1$ missing edges between parts according to the ordering. We claim that $e_\ell'$ comes before $e_{K_0+\ell}$. Indeed, let $e_\ell'$ be between $V_i$ and $V_j$ and $e_{K_0+\ell}$ be inside $V_m$. If $e_\ell'$ comes after $e_{K_0+\ell}$, then $G[U]$ contains $e_{K_0+\ell}$ edges inside parts and there are less than $\ell$ edges missing between parts, contradicting Claim \ref{aob}. 

Apart from a $o(n^{r-1})$ term, the number of copies of $S_r$ added to $T$ by having the edges $e_{K_0+1},\dots, e_{K_0+K_2}$ is at most the number of copies of $S_r$ that are removed from $T$ by removing the edges $e_1',\dots,e_{K_2}'$. Apart from a $o(n^{r-1})$ term, the number of copies of $S_r$ added to $T$ by having the edges $e_{1},\dots, e_{K_0}$ is at most the number of copies of $S_r$ we would get by adding $K_0$ edges inside $V_1$. Finally, by removing $e_{K_2+1}'$, we removed $\Theta(n^{r-1})$ copies of $S_r$. Therefore, $\cN(S_r,G)<\cN(S_r,G')-\Theta(n^{r-1})$, where $G'$ is obtained from $T$ by adding a $\cD(F)$-free graph with $K_0$ edges inside $V_1$. This is a contradiction when considering $\ex(n,S_r,F)$. In the case we consider $\ex_r(n,F)$, we also obtain a contradiction using Proposition \ref{exac}, since the smaller stars containing any $e_i$ or $e_i'$ give a $O(n^{r-2})$ term.

Let us assume now that $G$ contains less than $K_0$ edges inside parts. Then clearly these extra edges create less copies of $S_r$ than placing $K_0$ edges inside $V_1$ by $\Theta(n^{r-1})$, thus $\cN(S_r,G)<\cN(S_r,G')-\Theta(n^{r-1})$ and this is a contradiction as above.

Finally, assume that $G$ contains exactly $K_0$ edges inside parts. If they are in the same part, we are done. If they are in at least two parts and there is a missing edge between parts, then again $\cN(S_r,G)<\cN(S_r,G')-\Theta(n^{r-1})$ and we are done. If there are no missing edges between parts, take each vertex incident to an edge inside a part and add further vertices such that altogether we take $\max\{n_0,2K_0\}$ vertices from each part. The resulting graph contains $\ex(k\max\{n_0,2K_0\},F)=T(k\max\{n_0,2K_0\},k)+K_0$ edges on $k\max\{n_0,2K_0\}$, which contradicts our assumption on the uniqueness of the extremal graph for $\ex(n,F)$.
\end{proof}

\vskip 0.3truecm

\textbf{Funding}: Research supported by the National Research, Development and Innovation Office - NKFIH under the grants FK 132060 and KKP-133819.

\vskip 0.3truecm


\end{document}